\documentclass[11pt,letterpaper]{article}
\usepackage[margin=1in]{geometry}
\usepackage[utf8]{inputenc}
\usepackage[dvipsnames]{xcolor}
\usepackage[bookmarksnumbered,linktocpage,hypertexnames=false,colorlinks=true,linkcolor=NavyBlue,urlcolor=NavyBlue,citecolor=ForestGreen,anchorcolor=green,breaklinks=true,pagebackref=true,pdfusetitle]{hyperref}
\usepackage{bbm,braket,microtype,amsmath,amssymb,amsthm,amsfonts,dsfont,dynkin-diagrams,mathtools,colortbl,booktabs,algorithm,algpseudocode,graphicx,enumitem,xspace,float,mathdots,ellipsis,caption,subcaption,mleftright,multirow,adjustbox}
\usepackage[normalem]{ulem}
\usepackage[T1]{fontenc}
\usepackage{textcomp}
\usepackage[english]{babel}
\usepackage[capitalize,nameinlink]{cleveref}
\usepackage{tikz}
\usepackage{stmaryrd}
\usepackage{listings}
\usepackage{thm-restate}
\lstdefinestyle{sagestyle}{
    language=Python,
    basicstyle=\ttfamily\small,
    keywordstyle=\color{blue},
    commentstyle=\color{gray},
    stringstyle=\color{teal},
    showstringspaces=false,
    breaklines=true,
    frame=single,
    columns=fullflexible,
    keepspaces=true
}

\usepackage[numbers]{natbib}

\allowdisplaybreaks[4]
\newtheorem{theorem}{Theorem}[section]
\newtheorem{corollary}[theorem]{Corollary}\AddToHook{env/corollary/begin}{\crefalias{theorem}{corollary}}
\AddToHook{env/definition/begin}{\crefalias{theorem}{definition}}
\newtheorem{lemma}[theorem]{Lemma}\AddToHook{env/lemma/begin}{\crefalias{theorem}{lemma}}
\crefname{lemma}{Lemma}{Lemmata}
\AddToHook{env/proposition/begin}{\crefalias{theorem}{proposition}}\crefname{proposition}{Proposition}{Propositions}
\AddToHook{env/problem/begin}{\crefalias{theorem}{problem}}
\AddToHook{env/assumption/begin}{\crefalias{theorem}{assumption}}
\AddToHook{env/claim/begin}{\crefalias{theorem}{claim}}
\AddToHook{env/conjecture/begin}{\crefalias{theorem}{conjecture}}
\newtheorem{remark}[theorem]{Remark}\AddToHook{env/remark/begin}{\crefalias{theorem}{remark}}
\newtheorem{example}[theorem]{Example}\AddToHook{env/example/begin}{\crefalias{theorem}{example}}
\AddToHook{env/question/begin}{\crefalias{theorem}{question}}
\newtheorem*{theorem*}{Theorem}
\newtheorem*{corollary*}{Corollary}
\newtheorem*{proposition*}{Proposition}
\AddToHook{cmd/appendix/after}{\crefalias{section}{appendix}}
\numberwithin{equation}{section}

\crefname{lstlisting}{listing}{listings}
\Crefname{lstlisting}{Listing}{Listings}

\usepackage{aliascnt}
\makeatletter
\let\c@algorithm\relax 
\makeatother
\newaliascnt{algorithm}{theorem}

\makeatletter
\let\c@table\relax 
\makeatother
\newaliascnt{table}{theorem}

\newcommand{\PP}{\mathbb{P}}

\newcommand{\CC}{\mathbb{C}}
\newcommand{\ZZ}{\mathbb{Z}}

\newcommand{\cR}{\mathcal{R}}

\newcommand{\cZ}{\mathcal{Z}}

\newcommand{\cI}{\mathcal{I}}

\newcommand{\ff}{\ensuremath{\boldsymbol{f}}}
\newcommand{\hh}{\ensuremath{\boldsymbol{h}}}

\newcommand{\ft}{\mathfrak{t}}
\newcommand{\fp}{\mathfrak{p}}

\DeclareMathOperator{\length}{length}

\DeclareMathOperator{\codim}{codim}
\DeclareMathOperator{\rk}{rk}

\DeclareMathOperator{\Proj}{Proj}
\DeclareMathOperator{\gm}{gm}
\DeclareMathOperator{\am}{am}
\DeclareMathOperator{\sm}{sm}
\DeclareMathOperator{\Res}{Res}
\DeclareMathOperator{\T}{T}
\DeclareMathOperator{\mult}{mult}
\DeclareMathOperator{\red}{red}

\title{Resultant multiplicity via projective degrees\\ and applications to tensor eigenvalues}
\author{
M. Levent Do\u{g}an 
\thanks{Ludwig-Maximilians-Universit\"{a}t M\"{u}nchen, Germany. \url{mahmut.dogan@lmu.de} }
\and 
Elias Tsigaridas 
\thanks{Inria Paris and Sorbonne University, France.
\url{elias.tsigaridas@inria.fr}
}
\and 
Zafeirakis Zafeirakopoulos \thanks{Department of Mathematics, National and Kapodistrian University of Athens, Greece. \url{zzafeir@math.uoa.gr}}}
\date{\small\today}

\begin{document}

\maketitle

\begin{abstract}
Given a system $\ff=(f_1,\ldots,f_n)$ of $n$ homogeneous forms in $n$ variables of the same degree, Macaulay's resultant vanishes precisely when the polynomials have a common projective zero.
Its order of vanishing measures the singularity of the resultant hypersurface at $\ff$.
In this paper, we study how this multiplicity reflects the geometry of the projective zero scheme defined by $\ff$.

We give an exact formula for the multiplicity, expressed in terms of the projective degrees of the rational map defined by $\ff$.
As a consequence, we obtain a geometric lower bound involving the degrees, dimensions, and multiplicities of the irreducible components of the projective zero scheme.
This extends the multiplicity estimates of Roy and Ghidelli from zero-dimensional schemes to schemes of arbitrary dimension.

Finally, we apply this geometric estimate to tensor eigenvalues.
It translates directly into a lower bound for the algebraic multiplicity of a tensor eigenvalue in terms of the geometry of its eigenscheme.
This settles a conjecture by Canino et al. and consequently settles earlier conjectures of Qi and of Hu and Ye concerning the relationship between algebraic, geometric, and span multiplicities of tensor eigenvalues. 
\end{abstract}

\tableofcontents

\section{Introduction}

Let $n\geq 2, d\geq 1$, and let \[
R\coloneqq \CC[x_1,\ldots,x_n],\qquad \ff=(f_1,\dots,f_n)\in R_d^n,
\] 
where $R_{d}$ denotes the vector space of homogeneous polynomials of degree $d$.
The Macaulay resultant \cite{Macaulay-02,Macaulay-16} is an irreducible polynomial in the coefficients of $\ff$, uniquely characterized up to a constant (which is irrelevant for our discussion) by \[
\Res(\ff)=0 \quad\text{if and only if}\quad f_1,\ldots, f_n \,\text{ have a common zero in }\PP^{n-1}.
\] That is, $\Res$ characterizes precisely the homogeneous square systems with a common projective zero.
We refer to \cite{Jouanolou-91,Cox-Little-Oshea-05,GKZ-08} for the construction and basic properties of the resultant.

The resultant defines a hypersurface in the affine space \[
\cR \coloneqq \{\ff\in R_d^n\mid \Res(\ff)=0\}\subset R_d^n
\] which we call the \emph{resultant hypersurface}.
A general point of this hypersurface is a system with exactly one simple projective zero.
Indeed, the incidence variety $\mathcal{I}=\{(\ff,[p])\mid \ff(p)=0\}$ is irreducible and maps generically one-to-one onto $\cR$; the locus of systems with two distinct zeros or a nonreduced zero has codimension at least two in $R_d^n$.

In fact, a closer look at $\mathcal{I}$ and its projection onto $\cR$ gives the following elegant characterization of the smooth points of the resultant hypersurface: \[
\ff\text{ is a smooth point of }\cR\quad\text{if and only if}\quad f_1,\dots,f_n\text{ have a unique, simple zero in }\PP^{n-1},
\] so systems with multiple or non-reduced zeros lie in the singular locus of the resultant hypersurface, see \cite[Proposition~4.2]{GGTV-23}.
For binary forms, this also follows from \cite[Chapter~3, Corollary~3.7]{GKZ-08}, which additionally provides a formula for the coordinates of the unique zero in terms of the partial derivatives of the resultant.

This observation leads to our central question:
\emph{how much of the geometry of the projective zero scheme $X\subset\PP^{n-1}$ of $\ff$ is described by the order of vanishing of $\Res$ at $\ff$}?
We write \[
\mu(\ff) \coloneqq \mult_{\ff}(\Res)
\] for the multiplicity of the resultant at $\ff$.
That is, $\mu(\ff)$ is the least total degree of a non-zero term in the Taylor expansion of $\Res(\ff+\hh)$ in the coefficients of $\hh$.
From the previous discussion it follows that $\mu(\ff)=0$ if and only if $X$ is the empty scheme, and $\mu(\ff)=1$ if and only if $X$ is a simple point.

We begin with two elementary examples illustrating how $\mu(\ff)$ reflects the geometry of $X$.

\begin{example}[Linear systems]
\label{ex:linear}
Let \[
f_\ell=c_{\ell,1}x_1+\dots+c_{\ell,n}x_n,\qquad \ell=1,\ldots,n,
\] be a system of $n$ linear forms in $n$ variables, and let $C=[c_{i,j}]_{i,j\in [n]}$ be its coefficient matrix.
Then, $\Res(\ff)=\det C$.
If $\rk C=r$, then $X=\PP(\ker C)\cong\PP^{n-r-1}$, while \[
\mu(\ff) = n-r = \dim X + 1,
\] since order $n-r-1$ derivatives of $\det$ either identically vanish, or they are minors of size $r+1$ (up to a sign), which all vanish on a matrix of rank $r$.
Hence, in this example, the multiplicity detects the dimension of the projective zero scheme.
\end{example}

\begin{example}[Binary forms]
\label{ex:bivariate}
Suppose $n=2$, and $
f_1, f_2$ are two binary forms of degree $d$.
In this case, $\Res(f_1,f_2)$ is the determinant of the \emph{Sylvester matrix}, whose entries are the coefficients of $f_1,f_2$, see \cite[Chapter~3,\S 1]{Cox-Little-Oshea-05}.
By \cite[Theorem~1]{Barnett-80}, the corank of the Sylvester matrix equals $N\coloneqq \deg\gcd(f_1,f_2)$.
Hence, the partial derivatives of the resultant all vanish up to order $N$.
Consequently, for binary forms we have \[
\mu(\ff)\geq \deg\gcd(f_1,f_2).
\]
Thus, the multiplicity of the resultant at a system of binary forms measures the degree of its projective zero scheme.
\end{example}

When the projective zero scheme is zero-dimensional, the results of Roy \cite{Roy-13} and Ghidelli \cite{Ghidelli-19} give the first general relation between its geometry and resultant multiplicity.
\begin{theorem}[{[\citealp[Theorem~5.2]{Roy-13};
 \citealp[Theorem~0.1]{Ghidelli-19}]}]
\label{thm:roy-ghidelli}
    Let $\ff=(f_1,f_2,\dots,f_n)$ be a system of $n$ homogeneous forms in $n$ variables of degree $d$ whose projective zero scheme has degree $N$.
    Then, Macaulay's resultant vanishes at $\ff$ with multiplicity at least $N$:\begin{equation}
    \label{eq:roy-ghidelli}
    \mu(\ff)\geq N \quad \text{if} \quad f_1,\dots,f_n \text{ have }N\text{ projective zeros, counted with multiplicity.}
    \end{equation}
\end{theorem}

For binary forms, \cref{thm:roy-ghidelli} follows from \cref{ex:bivariate}. 
We note that the inequality \eqref{eq:roy-ghidelli} need not be an equality when $n\geq 3$.
For example, the system $\ff=(x^2,xy,y^2)$ over $\CC[x,y,z]$ defines a zero-dimensional projective scheme of degree $3$ (the point $[0:0:1]$ with multiplicity $3$), whereas the resultant vanishes at $\ff$ with multiplicity $4$.

In this paper, we provide an exact formula for $\mu(\ff)$ in terms of the projective degrees of the rational map associated with $\ff$, see \cref{thm:main-resultant}. 
By relating the projective degrees to the geometry of the projective zero scheme, we obtain a lower bound for the multiplicity in terms of the degrees, multiplicities and dimensions of the irreducible components, see \cref{cor:main-resultant}.

\subsection{Resultant multiplicity via projective degrees}
Let $I=(f_1,\ldots,f_n)$ be the ideal of $R$ generated by $\ff$, and let \[
X \coloneqq \Proj(R/I)
\] be its projective zero scheme.
We choose $n-1$ general linear combinations \[
g_{\ell} = \sum_{j=1}^n a_{\ell j} f_j, \qquad \ell=1,\ldots,n-1, \; a_{\ell,j}\in\CC.
\]
For $0\leq \ell\leq n-1$, we set \[
K_\ell \coloneqq (g_1,\ldots,g_\ell),\qquad J_\ell \coloneqq K_\ell : I^{\infty},\qquad Y_\ell \coloneqq \Proj(R/J_\ell),
\] 
where $K_0\coloneqq (0)$.
Thus, $Y_\ell$ is the scheme-theoretic closure of $
\cZ_{\PP}(g_1,\ldots,g_\ell)\setminus X.
$ We write \[
y_\ell\coloneqq \deg Y_\ell,
\] with the convention that the empty scheme has degree zero.

We note the following interpretation:
If $\ff\neq 0$, then it defines a rational map \[
\phi_{\ff}:\PP^{n-1}\dashrightarrow \PP^{n-1},\qquad [p]\mapsto [f_1(p):\cdots:f_n(p)]
\] with base scheme $X$, and $y_0,\ldots,y_{n-1}$ are its \emph{projective degrees}: 
$Y_\ell$ is the closure of the inverse image of a general linear subspace of codimension $\ell$, and $y_\ell$ is its degree.
See \cite[Example~19.4]{Harris-93} for the definition of the projective degrees,
and \cref{lem:projective-degrees} (or \cite[7.1.3]{Dolgachev-12}) for an equivalent reformulation in terms of the class of the graph of $\phi_{\ff}$ in the Chow ring of $\PP^{n-1}\times\PP^{n-1}$.

The first result of this paper is an exact formula for resultant multiplicity in terms of $y_\ell$'s.

\begin{restatable}[Resultant multiplicity formula]{theorem}{resultant}
\label{thm:main-resultant}
For every $\ff\in R_d^n$, \[
\mu(\ff) = \sum_{\ell=0}^{n-1} d^{n-\ell-1} \left( d^\ell - y_\ell\right).
\]
\end{restatable}

If $\ff\neq 0$, then $y_0=1$, so the term with $\ell=0$ vanishes and the sum may start at $\ell=1$.
If $\ff=0$, then all the $Y_\ell$'s are empty, and \cref{thm:main-resultant} yields \[
\mu(0) = n d^{n-1},
\] which is the total degree of the resultant, see \cite[(3.1)~Theorem]{Cox-Little-Oshea-05}.

As an example, let us consider the system $\ff=(x_1 q,x_2q,\dots,x_n q)$ where $q$ is a general form of degree $d-1$ with $d\geq 2$.
In this case, $X=\cZ_{\PP}(q)\subset\PP^{n-1}$ is a hypersurface and the Roy--Ghidelli bound does not apply.
Outside $X$, the zero scheme of $g_1,\dots,g_\ell$ is a general linear subspace of codimension $\ell$, for $\ell=1,\dots,n-1$.
Hence, we have $y_\ell=1$ and
\cref{thm:main-resultant} yields \[
\mu(\ff) = nd^{n-1} - \sum_{\ell=0}^{n-1} d^{\ell}.
\]

A notable feature of the proof of \cref{thm:main-resultant} is that it is entirely geometric: beyond the degree of the resultant, it uses only its characterization as the defining polynomial of the locus of systems with a common projective zero, and does not invoke its determinantal construction.

\subsection{Geometric estimate}
The resultant multiplicity formula (\cref{thm:main-resultant}) via projective degrees is exact, but its relationship with the geometry of the projective zero scheme is not transparent.
We now obtain a directly geometric, though non-exact, lower bound for the multiplicity.

As we have seen previously in \cref{ex:bivariate,ex:linear}, and from \cref{thm:roy-ghidelli}, the multiplicity of the resultant can detect both the dimension and the degree of $X$ in special cases.
Our next result generalizes these observations and gives a lower bound for $\mu(\ff)$ directly in terms of $X$.
It states that each irreducible component $X_k$ of the zero scheme contributes at least \[
\mult_{X_k}(X)\deg(X_k)\, (\dim(X_k)+1) \, d^{\dim (X_k)}
\] to $\mu(\ff)$, where $\mult_{X_k}(X)$ denotes the multiplicity of $X$ along $X_k$.

\begin{restatable}[Geometric multiplicity bound]{corollary}{lowerbound}
\label{cor:main-resultant}
    Let $X_1,\dots,X_c$ be the irreducible components of~$X$.
    Then, \[
    \mu(\ff)\geq \sum_{k=1}^c \mult_{X_k}(X)\deg(X_k)\, (\dim(X_k)+1)\, d^{\dim (X_k)}.
    \]
\end{restatable}

When $X$ is a zero-dimensional scheme, \cref{cor:main-resultant} yields \[
\mu(\ff)\geq \sum_{k=1}^c \mult_{X_k}(X)\deg(X_k) = \deg(X),
\] and we recover \cref{thm:roy-ghidelli}.
This also shows that the inequality in \cref{cor:main-resultant} can be strict since the Roy--Ghidelli bound need not be sharp.

Before we continue with the applications, we comment on the proof of \cref{cor:main-resultant}.
Define \[
z_r \coloneqq d y_{r-1}-y_r,\qquad r=1,\dots,n-1.
\] 
Then, \cref{thm:main-resultant} is equivalently stated as \[
\mu(\ff) = \sum_{r=1}^{n-1} (n-r) d^{n-r-1} z_r.
\] The number $z_r$ measures the degree lost in the $r$-th intersection $Y_{r-1}\cap \cZ_{\PP}(g_r)$ to the base locus $X$.
Some of this lost degree comes directly from the codimension-$r$ components of $X$, and we have \[
z_r \geq \sum\mult_{X_k}(Y_{r-1}\cap\cZ_{\PP}(g_r)) \, \deg(X_k) \geq  \sum\mult_{X_k}(X)\, \deg(X_k)
\] where the sum runs over all components $X_k$ with $\codim(X_k)=r$.
For such a component, we have \[
(n-r) d^{n-r-1} = (\dim(X_k)+1 )d^{\dim(X_k)}.
\]

\subsection{Tensor eigenvalues and multiplicities}

Tensor eigenvalues were introduced, in closely related forms, by Lim \cite{Lim-05} and  Qi \cite{Qi-05}.
Our notation and presentation follow \cite{GGTV-23,CFGL-26}.
Let $
\T \in \CC^n \otimes \left((\CC^{n})^\ast\right)^{\otimes d}
$ be a tensor of order $d+1$, where $(\CC^n)^\ast$ denotes the dual of the vector space $\CC^n$. 
We can view $\T$ as a linear map
\[
\T: (\CC^n)^{\otimes d}\rightarrow \CC^n.
\] 

The tensor $\T$ also induces a homogeneous polynomial map (that we also denote by $\T:\CC^n\rightarrow\CC^n$) via $\T(\omega) = \T(\omega^{\otimes d})$.
In standard coordinates, $\T$ is a polynomial map $\omega\mapsto (f_1(\omega),\dots,f_n(\omega))$ defined by $n$ homogeneous polynomials of degree $d$.
Conversely, every polynomial system $\ff$ of $n$ forms of degree $d$ defines a partially symmetric tensor $\T\in\CC^n\otimes \operatorname{Sym}^d((\CC^n)^\ast)$ via $\T(\omega)=(f_1(\omega),\dots,f_n(\omega))$.
We define the resultant of $\T$ as \[\Res(\T)\coloneqq\Res(f_1,\dots,f_n).
\] 

Let $\T,\ft\in \CC^n \otimes \left((\CC^{n})^\ast\right)^{\otimes d}$ be two tensors with $\Res(\ft)\neq 0$. 
We call a vector $\omega\in\CC^n\setminus\{0\}$ a \emph{$\ft$-eigenvector} of $\T$ with \emph{$\ft$-eigenvalue} $\lambda$ if \begin{equation}
\label{eq:eigenpair}
\T(\omega) = \lambda \ft(\omega).
\end{equation}
We denote by $E_{\T,\ft}(\lambda)$ the affine scheme defined by \cref{eq:eigenpair}.
We will call it the \emph{eigenscheme} corresponding to the $\ft$-eigenvalue $\lambda$.
Geometrically, the underlying set behind its reduced scheme is \[
\left(E_{\T,\ft}(\lambda)\right)_{\red} = \{\omega\in\CC^n\mid (\T-\lambda\ft)(\omega)=0\}.
\]
We define \[
\gm_{\T,\ft}(\lambda) \coloneqq \dim \left( E_{\T,\ft}(\lambda)\right)
\] to be the \emph{geometric multiplicity} of the $\ft$-eigenvalue $\lambda$, and \[
\sm_{\T,\ft}(\lambda)\coloneqq \dim\left( \operatorname{span}_{\CC}(E_{\T,\ft}(\lambda)_{\red})\right)
\] to be its \emph{span multiplicity}---the dimension of the linear span of the reduced scheme $\left(E_{\T,\ft}(\lambda)\right)_{\red}\subset\CC^{n}$.

When $d=1$, we can identify tensors in $ \CC^n\otimes(\CC^n)^\ast$ with matrices, i.e., linear maps $\CC^n\rightarrow\CC^n$. 
In this case, we have $\Res(\T)=\det(\T)$ and $\Res(\ft)=\det(\ft)$.
Since we assumed $\Res(\ft)\neq 0$, the $\ft$-eigenvalues coincide with the eigenvalues of the matrix $\ft^{-1}\T$.
Hence, the eigenscheme $E_{\T,\ft}(\lambda)$ is an eigenspace of $\ft^{-1}\T$, and geometric and span multiplicities coincide.

When $d>1$, however, $E_{\T,\ft}(\lambda)$ is not a linear subspace in general, since it is defined by a degree-$d$ polynomial system $\ff-\lambda \hh$, where $\T(\omega)=(f_1(\omega),\ldots,f_n(\omega))$ and $\ft(\omega)=(h_1(\omega),\ldots,h_n(\omega))$.
By homogeneity, $\ff-\lambda\hh$ defines a projective zero scheme whose affine cone is $E_{\T,\ft}(\lambda)$:
\[
X\coloneqq \cZ_{\PP}(\ff-\lambda \hh)
\]

In analogy with matrices, one can define a notion of \emph{algebraic multiplicity}, $\am_{\T,\ft}(\lambda)$, of a $\ft$-eigenvalue $\lambda$. 
Let us define the \emph{$\ft$-characteristic polynomial} of $\T$ by
\[ p_{\T,\ft}(s)\coloneqq\Res(\T-s \ft).
\]
We can see that $\ft$-eigenvalues coincide with the roots of $p_{\T,\ft}(s)$.
The algebraic multiplicity of a $\ft$-eigenvalue $\lambda$ is defined as the multiplicity of the $\ft$-characteristic polynomial at $s=\lambda$:\[
\am_{\T,\ft}(\lambda) \coloneqq \mult_{s=\lambda}(p_{\T,\ft}(s))
\]
When $d=1$, we have $\Res(\T- s\ft)=\det(\T-s\ft)=\det(\ft)\det(\ft^{-1}\T-s I_n)$.
Hence, $p_{\T,\ft}(s)$ coincides with the characteristic polynomial of the matrix $\ft^{-1}\T$ up to the constant $\det(\ft)$.

\begin{remark}
When $\ft=(x_1^d,\dots,x_n^d)$ as a polynomial system, $p_{\T,\ft}(s)=\Res(\T-s\ft)$ coincides with Canny's \emph{generalized characteristic polynomial} \cite{Canny-90}, defined as $
C(s) \coloneqq \Res( f_1 - s x_1^d, \dots, f_n - s x_n^d )$. 
Canny studied the lowest degree coefficient of $C(s)$ as an elimination operator, and he speculated that the order of vanishing of the generalized characteristic polynomial should be viewed as a ``measure of intersection multiplicities (in some appropriate sense)'' of the $f_i$'s.
\end{remark}

There are several conjectures in the literature on multiplicities of tensor eigenvalues.
Seeking a tensorial analogue of the linear algebra fact ``algebraic multiplicity is at least geometric multiplicity'',
Qi conjectured that the span multiplicity does not exceed the algebraic multiplicity:\begin{equation}
\label{eq:qi-conj}
\am_{\T,\ft}(\lambda) \geq \sm_{\T,\ft}(\lambda)\tag{\cite[Conjecture~1]{Qi-05}}
\end{equation}
Hu and Ye conjectured that if $E_1,E_2,\dots,E_c$ are the irreducible components of $E_{\T,\ft}(\lambda)$, then \begin{equation}
\label{eq:hu-ye}
\am_{\T,\ft}(\lambda)\geq \sum_{k=1}^c \dim(E_k) \, d^{\dim(E_k)-1}.\tag{\cite[Conjecture~1.1]{Hu-Ye-16}}
\end{equation} 
Since $\dim(E_k)=\dim\left(E_{\T,\ft}(\lambda)\right)$, for some $k$, this conjecture implies \begin{equation}
\label{eq:hu-ye-weak}
\am_{\T,\ft}(\lambda)\geq \gm_{\T,\ft}(\lambda)\, d^{\gm_{\T,\ft}(\lambda)-1}.
\end{equation}

Several special cases of the Hu--Ye conjecture were established for tensors that arise from uniform hypergraphs.
Cooper and Fickes \cite{Cooper-Fickes-22} confirmed \eqref{eq:hu-ye} for the zero eigenvalue of the adjacency tensor of a \emph{loose 3-hyperpath}.
Zheng \cite{Zheng-24} proved the weaker inequality \eqref{eq:hu-ye-weak} for nonzero eigenvalues of a $k$-uniform hyperpath.
Fan \cite{Fan-26} studied the Hu--Ye conjecture for weakly irreducible nonnegative tensors and confirmed it for the eigenvalues of $k$-uniform hypertrees, $k$-th powers of connected simple graphs and complete $3$-uniform hypergraphs on at least $4$ vertices, whose eigenvalue modulus agrees with the spectral radius.
We refer to \cite{Cooper-Dutle-12} for general information on eigenvalues of hypergraphs.

We note that the span and geometric multiplicities of an eigenvalue depend only on the reduced scheme $\left(E_{\T,\ft}(\lambda)\right)_{\red}$ rather than $E_{\T,\ft}(\lambda)$.
Hence, the conjectures mentioned above do not take into account the scheme structure of $E_{\T,\ft}(\lambda)$.
Very recently, Canino, Flavi, Galuppi and Luan \cite{CFGL-26} proposed two conjectures that relate algebraic multiplicity to the geometry of the eigenscheme.
Our focus will be on the following conjecture: 
\begin{equation}
\label{eq:cfgl-conjecture}
\am_{\T,\ft}(\lambda)\geq \sum_{k=1}^c \mult_{E_k}(E_{\T,\ft}(\lambda))\deg(E_k)\,\dim(E_k) \, d^{\dim(E_k) -1},\tag{\cite[Conjecture~1.5]{CFGL-26}}
\end{equation}
where $\mult_{E_k}(E_{\T,\ft}(\lambda))$ denotes the multiplicity of the affine scheme $E_{\T,\ft}(\lambda)$ along $E_k$ and $\deg(E_k)$ denotes the degree of the variety $E_k$.\footnote{In \cite{CFGL-26}, the authors define the degree of an irreducible component as the number of intersection points with a general linear subspace of complementary dimension, counted with multiplicity. 
In this paper, we always assume that irreducible components are reduced. Thus the quantity denoted by $\deg(E_k)$ in \cite{CFGL-26} is $\mult_{E_k}(E_{\T,\ft}(\lambda)) \deg(E_k)$ in our notation.}
Canino et al. observed that \eqref{eq:cfgl-conjecture} implies Hu and Ye's conjecture.
Indeed, since $\mult_{E_k}(E_{\T,\ft}(\lambda))\deg(E_k)\geq 1$ for a component $E_k$, \[
\mult_{E_k}(E_{\T,\ft}(\lambda))\deg(E_k)\dim(E_k)d^{\dim(E_k)-1}\geq \dim (E_k) d^{\dim(E_k) -1},
\] so \eqref{eq:hu-ye} follows from \eqref{eq:cfgl-conjecture}.

Canino et al. established \eqref{eq:cfgl-conjecture} in several cases (including when $n=2$ and when $E_{\T,\ft}(\lambda)$ is pure of dimension $n-1$) but left the general case open.
We prove it in full.
\begin{restatable}{theorem}{eigenvalue}
\label{thm:main-tensor}
Let $\T,\ft\in\CC^n\otimes((\CC^n)^\ast)^{\otimes d}$ be two tensors with $\Res(\ft)\neq 0$.
    Let $\lambda\in\CC$ and $E_1,E_2,\dots,E_c$ be the irreducible components of the eigenscheme $E_{\T,\ft}(\lambda)$.
    Then, \[
    \am_{\T,\ft}(\lambda) \geq \sum_{k=1}^c \mult_{E_k}(E_{\T,\ft}(\lambda))\;\deg(E_k)\; \dim(E_k)\; d^{\dim(E_k)-1}. 
    \]
\end{restatable}
\noindent
In particular, \eqref{eq:hu-ye} and \eqref{eq:cfgl-conjecture} hold.

Furthermore, we prove that \cref{thm:main-tensor} implies \eqref{eq:qi-conj}.
We call a projective variety (irreducible, reduced scheme) $X\subset\PP^{n-1}$ \emph{nondegenerate} if $X$ is not contained in any hyperplane.
Eisenbud and Harris \cite{Eisenbud-Harris-87} proved for a nondegenerate variety that \[
\deg(X) \geq \codim(X) + 1.
\] Note that $X$ is always nondegenerate in its linear span.
Hence, if $X\subset\PP^{n-1}$ is a projective variety and $E\subset \CC^{n}$ is the affine cone over $X$, then we have \[
\dim\left(\operatorname{span}_{\CC}(E)\right) \leq \deg(E) + \dim(E) - 1 \leq \deg(E)\dim(E), 
\] where the second inequality holds since $\deg(E)\geq 1$.
By applying this inequality to each irreducible component of the eigenscheme $E_{\T,\ft}(\lambda)$ we obtain:

\begin{restatable}{corollary}{spanmult}
\label{cor:span-mult}
    Let $\T,\ft\in\CC^n\otimes((\CC^n)^\ast)^{\otimes d}$ be two tensors with $\Res(\ft)\neq 0$, and
    let $\lambda\in\CC$.
    Then, \[
    \am_{\T,\ft}(\lambda)\geq \sm_{\T,\ft}(\lambda).
    \]
\end{restatable}
\noindent
In fact, if each component has dimension at least $N$, and multiplicity at least~$M$, then \cref{thm:main-tensor} yields \[
\am_{\T,\ft}(\lambda) \geq \sm_{\T,\ft}(\lambda)\; M \; d^{N-1}.
\] 

Finally, we note that \cref{thm:main-tensor} follows by restricting the resultant to the line $\T-s\ft$ and applying \cref{cor:main-resultant} to the projectivization of the eigenscheme $E_{\T,\ft}(\lambda)$. 

\subsection{Organization of the paper}

The preliminary \cref{sec:prelim} collects useful facts from algebraic geometry and intersection theory.
We give in \cref{lem:master} some properties of the varieties $Y_\ell$ and the projective degrees that will be used in the proof of \cref{cor:main-resultant}.
We also provide a formula in \cref{lem:projective-degrees} for the class of the graph of the rational map associated to $\ff$ in terms of the projective degrees.

In \cref{sec:main} we prove \cref{thm:main-resultant,cor:main-resultant}.
We compute the multiplicity, $\mu(\ff)$, by intersecting the resultant hypersurface with a general line through $\ff$.
We identify the intersection points as the coincidence points of the rational map associated to $\ff$ with a general regular map of degree $d$.
We compute the number of coincidences by intersecting their classes, whose coefficients in the Chow ring are identified with their projective degrees.
The geometric estimate follows by  calculating the degree lost at each intersection to the base locus.

Finally, we devote \cref{sec:eigenvalues} to multiplicities of tensor eigenvalues. 
We obtain \cref{thm:main-tensor} by a straightforward application of the geometric estimate to the eigenscheme.
\cref{cor:span-mult} follows at once using an inequality by Eisenbud and Harris that relates the dimension and degree of an irreducible projective variety to the dimension of its linear span.

\section{Preliminaries}
\label{sec:prelim}

Throughout, we work over the complex field $\CC$.
We set \[
R\coloneqq \CC[x_1,\ldots,x_n]
\] and denote by $R_d$ its degree-$d$ homogeneous component. 
We denote by $
\PP^{n-1} \coloneqq \PP(\CC^n)
$ the projective space of dimension $n-1$.

\subsection{Multiplicity of a polynomial at a point}

Let $m\geq 1$ and $F\in \CC[x_1,\dots,x_m]$ be a polynomial.
Its multiplicity at a point $p\in\CC^m$, denoted $\mult_p(F)$, is the least degree of a non-zero homogeneous term in the Taylor expansion of $F$ at $p$.
The first term of this expansion is $F(p)$, so we have $\mult_p(F)\geq 1$ if and only if $F(p)=0$.

The multiplicity at a point is captured by the multiplicity of $F$ restricted to a generic line through~$p$.
If $h(t)\in\CC[t]$ is a polynomial in one variable, we denote by $\operatorname{ord} h(t)$ the least degree of a non-zero term in $h(t)$, with the convention that $\operatorname{ord}(0)=\infty$.

\begin{lemma}
\label{lem:generic-line}
    For every $q\in\CC^m$ we have \[
    \operatorname{ord} F(p+tq) \geq \mult_{p}(F),
    \] and equality holds for generic $q$.
\end{lemma}
\begin{proof}
Let $N\coloneqq \mult_p(F)$.
    The Taylor expansion of $F$ at $p$ gives \[
    F(p+tq)= G_p(tq) + \text{higher degree terms}
    \] where $G_p\neq 0$ is a homogeneous polynomial of degree $N$.
    Hence, $G_p(tq)=t^{N}G_p(q)$.
    For a generic $q\in\CC^m$, we have $G_p(q)\neq 0$ so \[
    \operatorname{ord} F(p+tq) = N,
    \] with $\operatorname{ord} F(p+tq)> N$ if $G_p(q)=0$.
    This finishes the proof.
\end{proof}

\subsection{General linear combinations}

Let $I,J,K\subset R$ be ideals of $R$.
We define the \emph{saturation} of $K$ by $I$ as
\[
K:I^{\infty} \coloneqq \bigcup_{s\geq 1} K:I^s ,
\] where in general $K:J$ denotes the ideal quotient \[
K:J \coloneqq \{f\in R\mid f J \subset K\}.
\] 
If $I$ and $K$ are homogeneous ideals, then so is $K:I^{\infty}$.
Moreover, \[
\Proj\left(R / (K:I^{\infty})\right)
\] is the scheme-theoretic closure of $\Proj(R/K)\setminus \Proj(R/I)$.

Suppose $I\subset R$ is homogeneous and let \[
X\coloneqq \Proj(R/I)
\] be the projective scheme defined by $I$, let $Z$ be an irreducible component of $X$, and $\fp$ be the generic point of $Z$.
Ideal-theoretically, $\fp$ is the homogeneous minimal prime corresponding to $Z$.
The multiplicity of $X$ along $Z$ is defined as the length of the module $\mathcal{O}_{X,\fp}$ over the ring $\mathcal{O}_{\PP^{n-1},\fp}$, \[
\mult_{Z}(X) \coloneqq \length_{\mathcal{O}_{\PP^{n-1},\fp}}\mathcal{O}_{X,\fp},
\] where $\mathcal{O}_{\PP^{n-1},\fp}$ denotes the homogeneous localization of $R$ at the prime ideal $\fp$, i.e., \[
\mathcal{O}_{\PP^{n-1},\fp} \coloneqq R_{(\fp)} = \left\{\, \frac{f}{g}\,\mid\, f,g\in R \text{ homogeneous of the same degree}, g\not\in\fp\right\},
\]
and $\mathcal{O}_{X,\fp}\coloneqq (R/I)_{(\fp/I)}$.
We recall that the length of a module over a ring is the length of the longest chain of submodules.

Now, assume that $I=(f_1,\dots,f_n)$ where all the $f_i$'s are homogeneous of degree $d$, and choose general linear combinations \[
g_\ell = \sum_{j=1}^n a_{\ell j} f_j,\quad 1\leq \ell\leq n-1, \, a_{\ell j}\in \CC.
\] We define for $\ell=0,\dots,n-1$ \[
K_\ell \coloneqq (g_1,\dots,g_\ell),\qquad J_\ell \coloneqq K_\ell : I^\infty,\qquad Y_\ell \coloneqq \Proj(R/J_\ell),\qquad y_\ell\coloneqq\deg Y_\ell
\] with $K_0\coloneqq (0)$.
We will call $y_\ell,\, \ell=0,\dots,n-1$ the \emph{projective degrees}.

We recall that a sequence of polynomials $g_1,\dots,g_r$ in a ring $S$ is called a \emph{regular sequence} if $(g_1,\dots,g_r)\neq S$ and for each $\ell=1,\dots,r$, $g_\ell$ is a nonzerodivisor on $S/(g_1,\dots,g_{\ell-1})$.
See \cite[Chapter~17]{Eisenbud-95}.
If $S$ is a local ring of dimension $r$, with maximal ideal $\mathfrak{m}$, then $g_1,\dots,g_r\in\mathfrak{m}$ is called a \emph{system of parameters} if $\sqrt{(g_1,\dots,g_r)}=\mathfrak{m}$.

\begin{lemma}
\label{lem:master}
    Assume that $I\neq (0)$.
    \begin{enumerate}
        \item \label{it:master1}
        For each $\ell\geq 1$, for which $Y_{\ell-1}\neq\varnothing$, the element $g_\ell$ is a nonzerodivisor on $R/J_{\ell-1}$.
        Moreover, \begin{equation}
        \label{eq:saturation-eq}
        J_\ell = \left( J_{\ell-1}+(g_\ell) \right) : I^\infty.
                \end{equation}
                Consequently, $Y_\ell$ is the closure of $(Y_{\ell-1}\cap \cZ_{\PP}(g_\ell))\setminus X$.
        \item \label{it:master2}
        If $\operatorname{ht}(I)\geq r$ (equivalently, if $\codim X\geq r$), then $g_1,\dots,g_r$ is a regular sequence.
        In this case, we have $y_\ell=d^\ell$ for every $\ell< r$.
        \item \label{it:master3}
        Let $Z$ be an irreducible component of $X$ of codimension $r$, let $\fp$ be its generic point, and \[
        S \coloneqq \mathcal{O}_{\PP^{n-1},\fp}
        \] be the corresponding local ring.
        Then, $S$ is a regular local ring of dimension $r$ with $g_1,\dots,g_r$ forming a regular sequence, and a system of parameters.
        Moreover, for every $0\leq j<r$  \[
        J_j S = (g_1,\dots,g_j) S.
        \]
        Consequently, the multiplicity of $Y_{r-1}\cap\cZ_{\PP}(g_r)$ along $Z$ is  \[
        \lambda\coloneqq \length_{S}\left(S/(g_1,\dots,g_r)S\right) < \infty,
        \] and \[
        \lambda \geq \mult_Z(X).
        \]
    \end{enumerate}
\end{lemma}
\begin{proof}
By \cite[Exercise~17.3]{Eisenbud-95}, the set of zerodivisors of $R/J_{\ell-1}$ coincides with the union of associated primes of $R/J_{\ell-1}$.
    Since $J_{\ell-1}$ is $I$-saturated, the associated primes of $R/J_{\ell-1}$ cannot contain $I$.
    Hence, for an associated prime $\mathfrak{q}$ of $R/J_{\ell-1}$, the collection of coefficients $(a_1,\dots,a_n)\in\CC^n$ with $\sum_{i=1}^n a_i f_i\in \mathfrak{q}$ form a proper linear subspace of $\CC^n$.
    Since $R/J_{\ell-1}$ have only finitely many associated primes, a general linear combination of the $f_i$'s avoids all associated primes, and hence it is a nonzerodivisor.
    \Cref{eq:saturation-eq} follows from \[
    J_\ell = K_\ell : I^{\infty} = (K_{\ell-1}+(g_\ell)):I^{\infty} = (J_{\ell-1} + (g_\ell)) : I^{\infty}
    \] since $J_{\ell-1}=K_{\ell-1}:I^{\infty}$.
    This proves the first assertion.

    We now prove the second assertion.
    Let $1\leq\ell \leq r$.
    For an associated prime $\mathfrak{q}$ of $R/K_{\ell-1}$, we have $\mathrm{ht}(\mathfrak{q})=\codim R/K_{\ell-1}=\ell-1$.
    On the other hand, $\mathrm{ht}(I)\geq r$ so we have $\mathrm{ht}(\mathfrak{q})< \mathrm{ht}(I)$.
    This shows that $I\not\subset \mathfrak{q}$.
    As in the proof of the first assertion, a general linear combination of the $f_i$'s avoids all associated primes of $R/K_{\ell-1}$, and hence it is a nonzerodivisor.
    This shows that $g_1,\dots,g_r$ is a regular sequence.
    The assertion $y_\ell=d^{\ell}$ then follows from the B\'{e}zout theorem since $K_\ell=J_\ell$ and $Y_\ell=\Proj(R/K_{\ell})$ for $\ell < r$.

    We now prove the third assertion.
    Since $R$ is a polynomial ring, its localizations at prime ideals are regular local rings (\cite[Corollary~19.14]{Eisenbud-95}).
    This shows that $S$ is a regular local ring.
    Since $Z$ is an irreducible component of $X$, the prime $\fp$ is minimal over $I$ in $R$.
    Therefore, \[
    \sqrt{IS} = \fp S.
    \] In particular, $IS$ is $\fp S$-primary, and \[
    \length_S (S/IS) = \mult_Z(X).
    \] 
    
    We now prove by induction that $g_1,\dots,g_r$ form a regular sequence.
    The proof is similar to the proof of the second assertion:
    Let $1\leq j \leq r$, and assume that $g_1,\dots,g_{j-1}$ form a regular sequence.
    Let $S'\coloneqq S/(g_1,\dots,g_{j-1})S$, and $\mathfrak{q}$ be an associated prime of $S'$.
    Since $S$ is Cohen--Macaulay and $g_1,\dots,g_{j-1}$ is a regular sequence, we have $\dim S'=\dim S-j+1=r-j+1>0$.
    Hence, $\fp S'$ is not an associated prime of $S'$, and $\mathfrak{q}\neq \fp S'$.
    Since $\sqrt{IS'}=\fp S'$, we have $IS'\not\subset \mathfrak{q}$.
    It follows that the images of the $f_i$'s in $S'$ do not all lie in $\mathfrak{q}$.
    As in the proof of the first assertion, we deduce that a general linear combination of $f_i$'s avoids all associated primes of $S'$.
    This proves that $g_j$ is a nonzerodivisor on $S'$.

Hence, $g_1,\dots,g_r$ form a regular sequence.
    Since $\dim S = r$, its quotient with $g_1,\dots,g_r$ gives \begin{equation}
   \label{eq:zero-dim}
    \dim(S/ (g_1,\dots,g_r)S)=0.
        \end{equation} Thus, $\sqrt{S(g_1,\dots,g_r)}=\fp S$, so $g_1,\dots,g_r$ form a system of parameters.

Let $1\leq j\leq r$.
We already proved that no associated prime of $S'=S/(g_1,\dots,g_{j-1})S$ contains $IS'$.
It follows that the ideal $S(g_1,\dots,g_j)$ is $IS$-saturated: \[
J_j S = ((g_1,\dots,g_j) : I^{\infty}) S = (g_1,\dots,g_j) S :(IS)^{\infty} = (g_1,\dots,g_j) S.
\]
In particular, for $j=r-1$, the ideal of $Y_{r-1}\cap\cZ_{\PP}(g_r)$ in $S$ is \[
(J_{r-1}+(g_r))S = (g_1,\dots,g_r)S.
\]
By \cref{eq:zero-dim}, $Z$ is an irreducible component of this intersection with multiplicity \[
\mult_Z( Y_{r-1}\cap\cZ_{\PP}(g_r) ) = \length_{S}(S/(g_1,\dots,g_r)S).
\]

We now prove $\lambda\geq \mult_Z(X)$.
    Since $g_1,\dots,g_r\in I$, there is a surjection \[
    S/(g_1,\dots,g_r)S \twoheadrightarrow S/IS.
    \] Taking lengths, we get \[
    \lambda = \length_{S}(S/(g_1,\dots,g_r)S) \geq \length_S S/IS = \mult_Z(X).
    \]
\end{proof}

\begin{corollary}
\label{lem:regular-degrees}
If $X=\varnothing$, then $y_\ell=d^{\ell}$ for every $\ell=0,\dots,n-1$.
\end{corollary}
\begin{proof}
We have $\mathrm{ht}(I)=n=\dim(R)$. 
Hence, \cref{lem:master}~(\ref{it:master2}) yields $y_\ell=d^\ell$ for every $1\leq \ell\leq n-1$.
For $\ell=0$, we have $Y_0=\PP^{n-1}$ so $y_0=1=d^{0}$.
\end{proof}

\subsection{Classes of graphs of rational maps in the Chow ring}

We now give a brief introduction to intersection theory. 
One of our aims is to prove a formula (\cref{lem:projective-degrees}) for the projective degrees, $y_\ell$, as the coefficients of the class of the graph in the \emph{Chow ring} of $\PP^{n-1}\times\PP^{n-1}$ of the rational map associated to $\ff$.
Our main reference in this section is the first two chapters of \cite{Eisenbud-Harris-16}.

Informally, the \emph{Chow group} $A(X)$ of a scheme $X$ is the set of all formal linear combinations of the equivalence classes $[Y]$ of subvarieties (irreducible, reduced subschemes) $Y\subset X$, modulo the rational equivalence.
Two subvarieties $Y_1,Y_2$ are called \emph{rationally equivalent} if there is a subvariety of $\PP^{1}\times X$ whose fibers over two distinct points are $\{t_0\}\times Y_1$ and $\{t_1\}\times Y_2$.
See \cite[Definitions~1.2 and 1.3]{Eisenbud-Harris-16}.

There is a multiplication operation on the Chow group that gives it a ring structure, called the Chow ring.
To define it, we need the following definitions:
Two subvarieties $Y_1,Y_2\subset X$ are said to intersect \emph{transversely} at a point $p\in Y_1\cap Y_2$ if $p$ is a smooth point of $X,Y_1,Y_2$ and the tangent spaces of $Y_1$ and $Y_2$ at $p$ span the tangent space of $X$: \[
T_p(Y_1) + T_p(Y_2) = T_p(X).
\] Two subvarieties are called \emph{generically transverse} if they intersect transversely at a generic point of each irreducible component of $Y_1\cap Y_2$.

\begin{theorem}[{\cite[Theorem-Definition~1.5]{Eisenbud-Harris-16}}]
\label{thm:transversal-intersection}
    Let $X$ be a smooth quasi-projective variety. 
    Then there is a unique product structure on $A(X)$ satisfying the condition:

    If two subvarieties $Y_1,Y_2\subset X$ are generically transverse, then \[
    [Y_1][Y_2] = [Y_1\cap Y_2].
    \] 
    This structure makes $A(X)$ into an associative, commutative ring, called the \emph{Chow ring} of $X$.
\end{theorem}

It is known that \[
A(\PP^{n-1}) = \ZZ[\alpha]/ (\alpha^n),
\] where $\alpha=[H]$ denotes the class of a hyperplane $H\subset\PP^{n-1}$.
Moreover, if $Y\subset\PP^{n-1}$ is a subvariety of degree $d$ and codimension $k$, then we have \[
[Y] = d[H]^k.
\] See \cite[Theorem~2.1]{Eisenbud-Harris-16} for both assertions.

We will need to make calculations in the Chow ring of the product of projective spaces. 
\begin{lemma}[{\cite[Theorem~2.10]{Eisenbud-Harris-16}}]
    \label{lem:chow-ring}
    We have \[
    A(\PP^{n-1}\times\PP^{n-1})\cong \ZZ[\alpha,\beta] / (\alpha^{n},\beta^n),
    \] where $\alpha = [H_1] = [H'_1\times\PP^{n-1}]$ and $\beta=[H_2]=[\PP^{n-1}\times H'_2]$ denote the classes of hyperplanes in the first and the second factor of $\PP^{n-1}\times\PP^{n-1}$, respectively.
    Moreover, the class of a hypersurface defined by a bihomogeneous polynomial of bidegree $(d,e)$ is $d\alpha+e\beta$.
\end{lemma}

Let $L_1,L_2\subset\PP^{n-1}$ be linear subspaces.
Then $L_1\times L_2$ is a subvariety of $\PP^{n-1}\times\PP^{n-1}$.
The second assertion of \cref{lem:chow-ring} implies that if $L_1$ and $L_2$ have codimensions $r$ and $s$, respectively, then \[
[L_1\times L_2] = [H_1]^{r} [H_2]^s.
\]

Let $0\neq\ff\in R_d^n$, and $\phi_{\ff}$ be the corresponding rational map, defined as \[
\phi_{\ff} : \PP^{n-1}\dashrightarrow \PP^{n-1},\quad [p]\mapsto [ f_1(p):\cdots : f_n(p)].
\]
It is regular on the open set $U\coloneqq \PP^{n-1}\setminus X$ where $X$ denotes the projective zero scheme of $\ff$.

Let us denote by $G_{\ff}$ the graph of $\phi_{\ff}$ and
by $\Gamma_{\ff}$ its closure in $\PP^{n-1}\times\PP^{n-1}$, namely, \[
G_{\ff}\coloneqq \left\{\,\left([p],\phi_{\ff}([p])\right)\, \mid \, [p]\in U\right\},\qquad
\Gamma_{\ff} \coloneqq \overline{G_{\ff}}.
\]
We identify $[\Gamma_{\ff}]$ in $A(\PP^{n-1}\times\PP^{n-1})$ in terms of the projective degrees:
\begin{lemma}
\label{lem:projective-degrees}
If $\ff\neq 0$, then
 $G_{\ff}$ and $\Gamma_{\ff}$ are irreducible of dimension $n-1$ and $\dim (\Gamma_{\ff}\setminus G_{\ff})\leq n-2$.
    Moreover, we have \begin{equation}
   \label{eq:class-formula}
    [\Gamma_{\ff}] = \sum_{\ell=0}^{n-1} \, y_\ell \, [H_1]^{\ell} \, [H_2]^{n-\ell-1}.
        \end{equation}
\end{lemma}
\begin{proof}
$G_{\ff}$ is isomorphic to $U$ via the map $[p]\mapsto ([p],\phi_{\ff}([p]))$, whose inverse is the coordinate projection onto the first coordinate.
Hence, $G_{\ff}$ is irreducible and has dimension $n-1$.
Since $\Gamma_{\ff}$ is the closure of $G_{\ff}$, it is also irreducible of dimension $n-1$ and \[
    \dim (\Gamma_{\ff}\setminus G_{\ff})\leq n-2.
    \]

    We now prove \Cref{eq:class-formula}.
    Since $\Gamma_{\ff}$ is irreducible and has dimension $n-1$, we have \[
    [\Gamma_{\ff}] = \sum_{\ell=0}^{n-1} \, c_\ell \, [H_1]^{\ell}\, [H_2]^{n-\ell-1}
    \] for some coefficients $c_\ell\in\ZZ$.
    By \cref{lem:chow-ring}, we can calculate $c_\ell$ as \[
    [\Gamma_{\ff}][H_1]^{n-\ell-1}[H_2]^\ell = c_\ell [H_1]^{n-1} [H_2]^{n-1}.
    \]
    We have $[H_1]^{n-\ell-1}[H_2]^{\ell}=[L_1\times L_2]$ where $L_1\subset\PP^{n-1}$ is a linear subspace of codimension $n-\ell-1$ and $L_2\subset\PP^{n-1}$ is a linear subspace of codimension $\ell$.
    Hence, by \cref{thm:transversal-intersection}, we have \[
    c_\ell = \deg\left(\, \Gamma_{\ff} \cap (L_1\times L_2)\right)
    \] when $L_1$ and $L_2$ are general linear subspaces of codimensions $n-\ell-1$ and $\ell$, respectively.

   Since $\dim (\Gamma_{\ff}\setminus G_{\ff})\leq n-2<\dim\Gamma_{\ff}$, a general subspace $L_1\times L_2$ intersects $\Gamma_{\ff}$ at $G_{\ff}$:\[
   \Gamma_{\ff}\cap (L_1\times L_2) = G_{\ff}\cap (L_1\times L_2).
   \]
   Now, the coordinate projection onto the first factor gives an isomorphism \[
   G_{\ff}\cap (L_1\times L_2) \cong U\cap L_1 \cap \phi_{\ff}^{-1}(L_2).
   \]
   Since $L_1$ is generic, we have \[
   c_\ell = \deg\left(\overline{U\cap \phi_{\ff}^{-1}(L_2)}\right).
   \]
   On the other hand, the variety on the right-hand side is precisely $Y_\ell$, so we have \[
   c_\ell=y_\ell
   \] and this finishes the proof.
\end{proof}

\section{Multiplicity of the resultant}
\label{sec:main}

\subsection{Formula via projective degrees}

In this section we prove \cref{thm:main-resultant}.
The proof operates as follows:
We calculate the resultant multiplicity by intersecting $\cR$ with a line through $\ff$ and a general system $\hh$ with $\Res(\hh)\neq 0$.
We identify the intersection points with the coincidence points of the maps $\phi_{\ff}$ and $\phi_{\hh}$.
These points arise from the intersection of the graphs of $\phi_{\ff}$ and $\phi_{\hh}$, and a calculation in the Chow ring of $\PP^{n-1}\times\PP^{n-1}$ yields a formula for the degree of their intersection in terms of the projective degrees.

We first consider the case $\ff=0$.
Since $\Res$ is a homogeneous polynomial of degree $nd^{n-1}$, \[
\mu(0) = n d^{n-1}.
\] On the other hand, if $\ff=0$, then $X=\PP^{n-1}$ and all the schemes $Y_\ell$ are empty so $y_\ell=0$ for every $\ell=0,\dots,n-1$.
Thus \[
\sum_{\ell=0}^{n-1} d^{n-\ell-1}(d^{\ell}-y_\ell) = \sum_{\ell=0}^{n-1} d^{n-\ell-1} d^{\ell} = \sum_{\ell=0}^{n-1}d^{n-1} = nd^{n-1}.
\] This proves \cref{thm:main-resultant} when $\ff=0$.
Henceforth, we assume $\ff\neq 0$.

Let $\hh\in R_d^n$ be a system with $\Res(\hh)\neq 0$, chosen generically so that all of the conditions used below hold.
It defines a regular map \[
\phi_{\hh}:\PP^{n-1}\rightarrow\PP^{n-1},\quad [p]\mapsto [h_1(p):\cdots : h_n(p)].
\]
We define the \emph{coincidence scheme} of $\phi_{\ff}$ and $\phi_{\hh}$ as \[
\mathcal{C}(\ff,\hh) \coloneqq \{[p]\in U \mid \phi_{\ff}([p])=\phi_{\hh}[p]\}, \qquad C(\ff,\hh)\coloneqq\deg \mathcal{C}(\ff,\hh).
\]

\begin{lemma}
\label{lem:coincidence-pts}
    For a general $\hh$ as above, the intersection \[
    \Gamma_{\ff}\cap \Gamma_{\hh}
    \] is contained in $G_{\ff}$, is transverse, and is therefore zero-dimensional and reduced.
    It is isomorphic to $\mathcal{C}(\ff,\hh)$ by projection onto the first coordinate.
    Moreover, \[
    C(\ff,\hh) = \sum_{\ell=0}^{n-1} d^{n-\ell-1} y_\ell.
    \]
\end{lemma}
\begin{proof}
Let \[
\partial \Gamma_{\ff}\coloneqq \Gamma_{\ff}\setminus G_{\ff}
\] denote the boundary of $\Gamma_{\ff}$.
By \cref{lem:projective-degrees} we have \[
\dim \partial\Gamma_{\ff} \leq n-2.
\]

Let \[
\mathcal{H}\coloneqq \{\hh\in R_d^n \mid \Res(\hh)\neq 0\} = R_d^n\setminus \cR,
\] and consider the incidence variety \[
\cI \coloneqq \{ ([p],[q],\hh)\in \partial \Gamma_{\ff}\times \mathcal{H} \mid \phi_{\hh}([p])=[q] \}.
\] For fixed $([p],[q])$, the equation $\phi_{\hh}([p])=[q]$ imposes $n-1$ independent linear conditions, so \[
\dim \cI = \dim \partial \Gamma_f + \dim \mathcal{H} - (n-1) \leq \dim R_d^n - 1.
\] Hence, for generic $\hh$, we have \[
\partial \Gamma_{\ff}\cap \Gamma_{\hh}=\varnothing 
\] which implies that \[ \Gamma_{\ff}\cap \Gamma_{\hh} = G_{\ff}\cap \Gamma_{\hh}.
\]

We now prove that $\mathcal{C}(\ff,\hh)$ is zero-dimensional and reduced.
Consider the map \[
\Psi : U\times \mathcal{H} \rightarrow \PP^{n-1}\times\PP^{n-1},\quad \Psi([p],\hh) = \left(\phi_{\ff}([p]), \phi_{\hh}([p])\right).
\]
Let \[
Y \coloneqq \Psi^{-1}(\Delta)
\] where $\Delta\subset\PP^{n-1}\times\PP^{n-1}$ is the diagonal.
Then, the fiber of $Y\rightarrow \mathcal{H}$ over $\hh\in \mathcal{H}$ is precisely the coincidence scheme $\mathcal{C}(\ff,\hh)$.

We claim that $\Psi$ is transverse to $\Delta$.
Fix $([p],\hh)\in Y$ and write \[
[q]\coloneqq \phi_{\ff}([p])=\phi_{\hh}([p]).
\]
For fixed $p$, the evaluation map \[
R_d^n \rightarrow \CC^n,\quad \delta \hh\mapsto \delta\hh(p)
\] is surjective. 
Moreover, the differential of the quotient map $\CC^n\setminus\{0\}\rightarrow \PP^{n-1}$ at $\hh(p)$ is the surjection \[
\CC^n \rightarrow \CC^{n}/\CC \hh(p) \cong T_{[q]}\PP^{n-1}.
\] 
It follows that, by varying $\hh$ in the open set $\mathcal{H}$, the point $[q]$ can move infinitesimally in every direction in $T_{[q]}\PP^{n-1}$.
In other words, the image of the differential of $\Psi$ at $([p],\hh)$ contains \[
\{0\}\oplus T_{[q]}\PP^{n-1}\subset \operatorname{im}\left(\mathrm{d} \Psi_{([p],\hh)}\right).
\]

On the other hand, the tangent space to $\Delta$ at $([q],[q])$ is \[
T_{([q],[q])} \Delta = \{(v,v)\mid v\in T_{[q]}\PP^{n-1}\}.
\] Hence, we have \[
T_{([q],[q])} \Delta + \left(\{0\}\oplus T_{[q]}\PP^{n-1} \right) = T_{[q]}\PP^{n-1}\oplus T_{[q]}\PP^{n-1}.
\]
This proves that $\Psi$ is transverse to $\Delta$.

It follows that $Y=\Psi^{-1}(\Delta)$ is smooth of codimension $n-1$ in $U\times \mathcal{H}$.
Since $\dim U=n-1$, \[
\dim Y = \dim \mathcal{H}.
\] Now, consider the projection \[
\pi: Y\rightarrow \mathcal{H}.
\] By generic smoothness, the generic fiber $\pi^{-1}(\hh)$ is smooth of dimension $\dim Y-\dim \mathcal{H}= 0$.
Hence, $\mathcal{C}(\ff,\hh)$ is zero-dimensional and reduced.

We now prove that $G_{\ff}$ and $\Gamma_{\hh}$ intersect transversely.
Let $([p],[q])\in G_{\ff}\cap \Gamma_{\hh}$.
Since $[p]\in U$, both graphs are smooth at $([p],[q])$, with tangent spaces \[
T_{([p],[q])} G_{\ff} = \{(v,\mathrm{d} \phi_{\ff}\mid_{[p]}(v) \mid v\in T_{[p]}\PP^{n-1}\},
\] and \[
T_{([p],[q])} \Gamma_{\hh} = \{(v,\mathrm{d} \phi_{\hh}\mid_{[p]}(v) \mid v\in T_{[p]}\PP^{n-1}\}.
\] The tangent space of the coincidence scheme $\mathcal{C}(\ff,\hh)$ at $[p]$ is \[
\ker ( \mathrm{d} \phi_{\ff}\mid_{[p]} - \mathrm{d} \phi_{\hh}\mid_{[p]} ).
\] Since $\mathcal{C}(\ff,\hh)$ is smooth and zero-dimensional, this kernel is zero.
Hence, \[
\mathrm{d} \phi_{\ff}\mid_{[p]} - \mathrm{d} \phi_{\hh}\mid_{[p]} : T_{[p]}\PP^{n-1}\rightarrow T_{[p]}\PP^{n-1}
\] is an isomorphism.
This proves that \[
T_{([p],[q])} G_{\ff} + T_{([p],[q])} \Gamma_{\hh} = T_{([p],[q])}\left( \PP^{n-1}\times\PP^{n-1}\right).
\] Thus, $G_{\ff}$ and $\Gamma_{\hh}$ intersect transversely.
As $\Gamma_{\ff}\cap\Gamma_{\hh}=G_{\ff}\cap \Gamma_{\hh}$, this shows, for generic $\hh$, that $\Gamma_{\ff}$ and $\Gamma_{\hh}$ intersect transversely at precisely $C(\ff,\hh)$ points.

Finally, we compute $C(\ff,\hh)$.
We have already proved \[
C(\ff,\hh) = \deg\left(\Gamma_{\ff}\cap\Gamma_{\hh}\right) =\#(\Gamma_{\ff}\cap\Gamma_{\hh}).
\] 
Since $\Res(\hh)\neq 0$, the projective degrees of $\phi_{\hh}$ are $d^\ell, \ell=0,\dots,n-1$ by \cref{lem:regular-degrees}.
Then, \cref{lem:projective-degrees} yields \begin{equation}
\label{eq:classes}
[\Gamma_{\ff}] = \sum_{\ell=0}^{n-1} y_\ell \, [H_1]^{\ell} \, [H_2]^{n-\ell-1}, \qquad [\Gamma_{\hh}] = \sum_{\ell=0}^{n-1} d^\ell \, [H_1]^{\ell} \, [H_2]^{n-\ell-1}.
\end{equation}

Since $\Gamma_{\ff}$ and $\Gamma_{\hh}$ intersect transversely, \cref{thm:transversal-intersection} gives $[\Gamma_{\ff}\cap\Gamma_{\hh}]=[\Gamma_{\ff}][\Gamma_{\hh}]$.   
   Then, by \cref{eq:classes} and \cref{lem:chow-ring} we have \[
    [\Gamma_{\ff}\cap \Gamma_{\hh}] = \sum_{\ell=0}^{n-1} d^{n-\ell-1}\, y_\ell [H_1]^{n-1} [H_2]^{n-1}.
    \]
    Thus, \[
    C(\ff,\hh)= \deg\left(\Gamma_{\ff}\cap \Gamma_{\hh}\right) = \sum_{\ell=0}^{n-1} d^{n-\ell-1} y_\ell,
    \] which finishes the proof.
\end{proof}

We now compute the multiplicity of $\Res$ at $\ff$ as the intersection multiplicity of the resultant hypersurface with a general line through $\ff$ (\cref{lem:generic-line}).
To this end, consider the projective line of systems \[
L \coloneqq \{[\ff - s \hh] \mid s\in \CC\}\cup \{[\hh]\} \subset \PP(R_d^n)
\] through $[\ff]$ and $[\hh]$.

\begin{lemma}
\label{lem:count}
We have \[
\mu(\ff) = n d^{n-1} - C(\ff,\hh).
\]
\end{lemma}
\begin{proof}
Since $L$ is a general line through $[\ff]$, its local intersection multiplicity with the resultant hypersurface at $[\ff]$ is precisely $\mu(\ff)$, see \cref{lem:generic-line}.

Let us denote by $\cR_{\PP}\subset\PP(R_d^n)$ the projectivization of the resultant hypersurface $\cR$.
After choosing $\hh$ generically, all the other intersections of $L$ with $\cR_{\PP}$ are transverse and belong to the open, dense subset of $\cR_{\PP}$ that consists of systems with a unique, simple projective zero.
Since the total degree of the resultant is \[
\deg(\Res) = n d^{n-1},
\] 
we have \[
nd^{n-1} = \mu(\ff) + \# \left((L\cap \cR_{\PP})\setminus \{[\ff]\}\right),
\] where the intersections counted in the second summand are simple.

Since $\Res(\hh)\neq 0$, the point $[\hh]$ does not belong to the resultant hypersurface, and, therefore, does not contribute to the intersection of $L$ and $\cR_{\PP}$.
Every other point of $L\setminus\{[\ff],[\hh]\}$ is of the form $[\ff-s\hh]$ for some unique, non-zero $s\in\CC^{\times}$.
For such $s$, the system $[\ff-s\hh]$ lies in $\cR_{\PP}$ if and only if \begin{equation}
    \label{eq:sol}
 (\ff - s \hh)(p)=0
\end{equation} for some $[p]\in\PP^{n-1}$. 
Since $\hh(p)\neq 0$, this implies $\ff(p)\neq 0$ and \cref{eq:sol} gives \[
\phi_{\ff}([p]) = \phi_{\hh}([p]). 
\]
Conversely, every solution of this equality determines a unique $s\in\CC^{\times}$ such that $(\ff-s\hh)(p)=0$, since we assumed that the intersections of $L$ and $\cR_{\PP}$ occur at systems with a unique, simple zero. 

For general $\hh$, these coincidence points are simple by \cref{lem:coincidence-pts}.
Hence,  \[
 \#\left((L\cap \cR_{\PP})\setminus \{[\ff]\}\right) = C(\ff,\hh).
\]
\end{proof}

We can now prove the main theorem. 

\resultant*
\begin{proof}
    By \cref{lem:coincidence-pts,lem:count}, \[
    \mu(\ff) = nd^{n-1}-\sum_{\ell=0}^{n-1} d^{n-\ell-1}y_\ell.
    \] Since \[
    nd^{n-1} = \sum_{\ell=0}^{n-1} d^{n-\ell-1} d^{\ell},
    \] we conclude that \[
    \mu(\ff) = \sum_{\ell=0}^{n-1} d^{n-\ell-1} (d^\ell - y_\ell).
    \] 
\end{proof}

\subsection{Geometric lower bound}

We now prove \cref{cor:main-resultant}.

\lowerbound*
\begin{proof}
If $\ff=0$, then $X=\PP^{n-1}$.
In this case we have \[
\mu(\ff)=\deg(\Res)=n d^{n-1} = (\dim(X)+1)d^{\dim X},
\] and the assertion holds with equality.
Hence, we will assume that $\ff\neq 0$.

    For $1\leq r\leq n-1$, define \[
    z_r \coloneqq d y_{r-1} - y_r.
    \] 
    We give a formula for $\mu(\ff)$ in terms of the $z_r$.
    Set $w_r\coloneqq d^r-y_r$ for $r=0,\dots,n-1$.
    Then, \[
    \begin{split}
    w_r &= d^{r}-y_r\\ &= d^r- (d y_{r-1}-z_r)\\
    &= d(d^{r-1}-y_{r-1})+z_r\\
    &= d w_{r-1} + z_r.
    \end{split}
    \] 
    By iterating, we obtain \begin{equation}
    \label{eq:y-z-sum}
    w_r = z_r + d z_{r-1} + \dots + d^{r-1} z_1.
    \end{equation}

    Now, by \cref{thm:main-resultant} and \cref{eq:y-z-sum}, \[
    \begin{split}
        \mu(\ff) &= \sum_{\ell=0}^{n-1} d^{n-\ell-1} \left(d^{\ell}-y_\ell\right)\\
        &=\sum_{\ell=0}^{n-1}d^{n-\ell-1}\sum_{r=1}^{\ell} d^{\ell-r}z_r\\
        &=\sum_{r=1}^{n-1} (n-r) d^{n-r-1} z_r.
    \end{split}
    \]
     To conclude, it is enough to prove that \begin{equation}
        \label{eq:claim}
        z_r \geq \sum_{\substack{k\\ \codim X_k=r}} \mult_{X_k}(X)\,\deg (X_k).
    \end{equation}
    Indeed, if \cref{eq:claim} holds, then we have $(n-r)d^{n-r-1}=(\dim(X_k)+1)d^{\dim(X_k)}$ for a codimension-$r$ component $X_k$, so \cref{eq:claim} yields \[
    \mu(\ff) = \sum_{r=1}^{n-1} (n-r) d^{n-r-1} z_r \geq \sum_{k=1}^c \mult_{X_k}(X)\deg(X_k)(\dim(X_k)+1) d^{\dim(X_k)}.
    \] 

    We now prove \Cref{eq:claim}.
    The interpretation of $z_r$ is that when we pass from $Y_{r-1}$ to $Y_{r}$, we intersect $Y_{r-1}$ with $g_r$, which has degree~$d$.
    By \cref{lem:master}~(\ref{it:master1}) and the B\'{e}zout theorem, the total degree of this intersection is $d y_{r-1}$.
    Some components of this intersection lie in $X$; and the other components constitute $Y_{r}$, and contribute $y_r$ to the degree of the intersection.
    Hence, $z_r$ is the total degree of the codimension-$r$ components of $Y_{r-1}\cap\cZ_{\PP}(g_r)$ that lie in $X$.
    Some of these components are the components of $X$.
   Hence, we have \[
    z_r \geq \sum_{\substack{k\\ \codim X_k=r}} \lambda_k \deg(X_k),\qquad \text{where }\;  \lambda_k \coloneqq \mult_{X_k}(Y_{r-1}\cap \cZ_{\PP}(g_r)). 
    \] 
    By \cref{lem:master}~(\ref{it:master3}), we have $\lambda_k\geq \mult_{X_k}(X)$ so \[
    z_r \geq \sum_{\substack{k\\ \codim X_k=r}} \mult_{X_k}(X) \deg(X_k).
    \]
    This proves \cref{eq:claim} and completes the proof.
\end{proof}

\section{Applications to tensor eigenvalues}
\label{sec:eigenvalues}

In this section we prove \cref{thm:main-tensor,cor:span-mult}.
The first result follows by a straightforward application of \cref{cor:main-resultant} to the eigenscheme $E_{\T,\ft}(\lambda)$.

\eigenvalue*
\begin{proof}
If $\lambda$ is not an eigenvalue, the assertion is trivial, so we assume that $\lambda$ is a $\ft$-eigenvalue and $\dim E_{\T,\ft}(\lambda)\geq 1$. 

Let $\ff$ and $\hh$ be the polynomial systems defined by $\T$ and $\ft$, respectively.
We first note that \[
\am_{\T,\ft}(\lambda) = \mult_{s=\lambda} \Res(\T-s\ft) = \operatorname{ord}\left( \Res(\T-\lambda\ft - s\ft)\right) = \operatorname{ord}\left( \Res(\ff-\lambda\hh - s\hh)\right),
\] where $\operatorname{ord}(\cdot)$ denotes the order of vanishing with respect to the variable $s$.
By \cref{lem:generic-line}, we have \[
\am_{\T,\ft}(\lambda)=\operatorname{ord}(\Res(\ff-\lambda\hh-s\hh))\geq\mult_{\ff-\lambda\hh}(\Res) = \mu(\ff-\lambda\hh).
\] Now, \cref{cor:main-resultant} yields \[
\mu(\ff-\lambda \hh) \geq \sum_{k=1}^c \mult_{X_k}(X) \, \deg(X_k) \, \left(\dim(X_k)+1\right)d^{\dim(X_k)},
\] where $X$ is the projective zero scheme of $\ff-\lambda \hh$.
Since $E_{\T,\ft}(\lambda)$ is the affine cone over $X$, we have \[
\mult_{X_k}(X) = \mult_{E_k}(E_{\T,\ft}(\lambda)),\quad \deg(X_k)=\deg(E_k),\quad \dim(X_k)=\dim(E_k)-1.
\] 
Hence, \[
\am_{\T,\ft}(\lambda) \geq \sum_{k=1}^c \mult_{E_k}(E_{\T,\ft}(\lambda))\, \deg(E_k)\, \dim(E_k)\, d^{\dim(E_k)-1}. 
\]
\end{proof}

We will use the following estimate of the dimension of the linear span of an irreducible variety in the proof of \cref{cor:span-mult}.

\begin{lemma}
\label{lem:deg-dim-ineq}
Let $E\subset\CC^n$ be an irreducible affine cone, and $V\coloneqq \operatorname{span}_{\CC}(E)$.
    Then, \[
    \dim(V)\leq \dim(E) + \deg(E) - 1.
    \]
\end{lemma}
\begin{proof}
Let $X\coloneqq \PP(E)\subset\PP(V)$.
 Since $V=\operatorname{span}_{\CC}(E)$, $X$ is not contained in any hyperplane in $\PP(V)$.
    By \cite[Proposition~0]{Eisenbud-Harris-87}, we have $\codim_{\PP(V)}(X)\leq \deg(X)-1$ which implies \[
   \dim(V) = \dim(X) + \codim_{\PP(V)}(X) + 1 \leq \dim(X) + \deg(X) = \dim(E) + \deg(E) -1.
    \]
\end{proof}

We can now prove \cref{cor:span-mult}.

\spanmult*
\begin{proof}
    If $\lambda$ is not a $\ft$-eigenvalue of $\T$, then $\am_{\T,\ft}(\lambda)=\sm_{\T,\ft}(\lambda)=0$.
    Hence, we will assume that $\lambda$ is an eigenvalue.
    
    Let $E_1,\dots,E_c$ be the irreducible components of $E_{\T,\ft}(\lambda)$, and define $V\coloneqq \operatorname{span}_{\CC}(E_{\T,\ft}(\lambda)_{\red})$ and $V_k\coloneqq\operatorname{span}_{\CC}(E_k)$.
    Since $E_{\T,\ft}(\lambda)_{\red}=\bigcup_{k=1}^c E_k$, we have $V = V_1 + V_2 + \dots + V_c$.
    Thus, \[
    \sm_{\T,\ft}(\lambda) = \dim(V) \leq \sum_{k=1}^c \dim(V_k).
    \]

    Now, \cref{lem:deg-dim-ineq} gives \[
    \dim(V_k)\leq \dim(E_k) + \deg(E_k)-1\leq \dim(E_k)\deg(E_k)
    \] where the second inequality holds because $\deg(E_k)\geq 1$.
    Summing over all components we get 
\[
    \sm_{\T,\ft}(\lambda) \leq \sum_{k=1}^c \dim(E_k)\deg(E_k).
\]

    By \cref{thm:main-tensor}, we have \[
    \am_{\T,\ft}(\lambda) \geq \sum_{k=1}^c \mult_{E_k}(E_{\T,\ft}(\lambda))\,\deg(E_k)\, \dim(E_k)\,  d^{\dim E_k-1}.
    \]
    Since $E_1,\dots,E_c$ are the irreducible components of $E_{\T,\ft}(\lambda)$, we have \[
    \mult_{E_{k}}(E_{\T,\ft}(\lambda))\geq 1, \qquad \dim(E_k)\geq 1, \qquad d^{\dim (E_k) -1}\geq 1.
    \]
    Thus, \[
    \am_{\T,\ft}(\lambda) \geq \sum_{k=1}^c \dim(E_k)\deg(E_k) \geq \sm_{\T,\ft}(\lambda).
    \] 
\end{proof}

\section*{Acknowledgements}

L. D. is supported by the European Union (ERC Grant SYMOPTIC, 101040907) and by the Deutsche Forschungsgemeinschaft (DFG, German Research Foundation, 556164098).
E. T. is  partially supported by the PGMO grant SOAP, 
the PHC PROCOPE project ``Quantum  games and polynomial optimization'',
and ANR PRC ZADyG (ANR-25-CE48-7058). 

\bibliographystyle{amsalpha}
\bibliography{references}

\end{document}